\documentclass[12pt, reqno]{amsart}
\usepackage[utf8]{inputenc}
\usepackage[english]{babel}
 
\usepackage[square,sort,comma,numbers]{natbib}

\usepackage[dvipsnames]{xcolor}
\usepackage{float}
\usepackage{pinlabel}
\usepackage{graphicx}
\usepackage{tikz}
\usetikzlibrary{arrows}
\usetikzlibrary{matrix,arrows,decorations.pathmorphing}

\usepackage{amsmath}
\usepackage{amssymb}
\usepackage{amsthm}

\usepackage{stmaryrd}
\usepackage[all]{xy}
\usepackage[mathcal]{eucal}
\usepackage{verbatim}\usepackage{fullpage} \usepackage{wrapfig}
\usepackage{lipsum}
\usepackage[all]{xy}
\theoremstyle{plain}
\newtheorem{thm}{Theorem}[section]

\newtheorem{lem}[thm]{Lemma}

\newtheorem{prop}[thm]{Proposition}
\newtheorem{cor}[thm]{Corollary}

\theoremstyle{definition}
\newtheorem{defn}[thm]{Definition}

\theoremstyle{remark}
\newtheorem*{rem}{Remark}

\newcommand{\nc}{\newcommand}
\nc{\dmo}{\DeclareMathOperator}

\DeclareMathOperator{\Diff}{Diff}

\DeclareMathOperator{\Fix}{Fix}

\DeclareMathOperator{\Mod}{Mod}
\DeclareMathOperator{\fix}{Fix}
\DeclareMathOperator{\Homeo}{Homeo}

\newcommand{\discColor}{Black}

\newcommand{\discSizeX}{2.1}

\nc{\para}[1]{\medskip\noindent\textbf{#1.}}

\title{On the non-realizability of braid groups by homeomorphisms}
\author{Lei Chen}

\begin{document}
\maketitle
\begin{abstract}
In this paper, we show that the projection $\Homeo_+(D^2_n)\to B_n$ does not have a section for $n\ge 6$; i.e., the braid group $B_n$ cannot be geometrically realized as a group of homeomorphisms of a disk fixing the boundary point-wise and $n$ marked points in the interior as a set. We also give a new proof of a result of Markovic \cite{Mar} that the mapping class group of a surface of genus $g$ cannot be geometrically realized as a group of homeomorphisms when $g\ge 2$.
\end{abstract}
\section{introduction}
Let $S_{g;m_1,...,m_r}^b$ be a surface of genus $g$ with $r$ sets of marked points and $b$ boundary components such that the $i$th set contains $m_i$ points. We omit the index $m_i$ and $b$ whenever they are zero. Let $\Homeo_+(S_{g;m_1,...,m_r}^b)$ be the group of orientation-preserving homeomorphisms of $S_{g;m_1,...,m_r}^b$ fixing $b$ boundary components point-wise and $r$ sets of points set-wise. Let $\Mod(S_{g;m_1,...,m_r}^b)$ be the \emph{mapping class group} of $S_{g;m_1,...,m_r}^b$; i.e.,
\[
\Mod(S_{g;m_1,...,m_r}^b):=\pi_0(\Homeo_+(S_{g;m_1,...,m_r}^b)).
\] There is an associated projection 
\[
pr_{g;m_1,...,m_r}^b:\Homeo_+(S_{g;m_1,...,m_r}^b)\to \Mod(S_{g;m_1,...,m_r}^b).
\]
In this paper, we establish the following result.
\begin{thm}\label{main}
The projections $pr_{0;n}, pr_{0;n,1}$ and $pr_{0;n}^1$ do not have sections for $n\ge 6$.
\end{thm}
The above theorem answers Question 3.11 in the survey of Mann--Tshishiku \cite{MT} and generalizes Salter--Tshishiku \cite{ST}. Let $\tau$ be the hyper-elliptic involution as in the following figure. 
\begin{figure}[h]
   \labellist 
  \small\hair 2pt
     \pinlabel \Large $\tau$ at 410 70
   \endlabellist
     \centerline{ \mbox{
 \includegraphics[width = 5in]{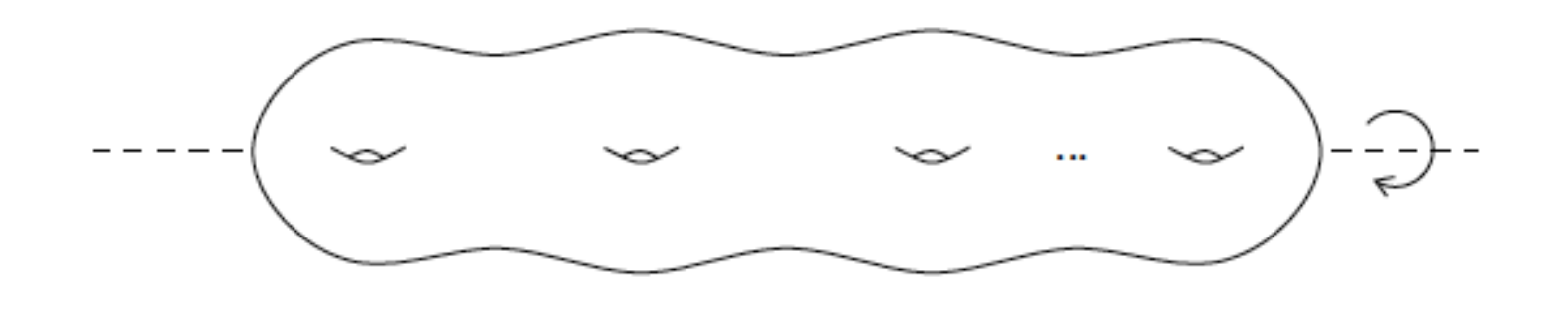}}}
 \caption{The hyper-elliptic involution $\tau$}
  \end{figure}

Let $\mathcal{H}_g<\Mod(S_g)$ be the \emph{hyper-elliptic mapping class group}, i.e., the centralizer of $\tau\in \Mod(S_g)$. Markovic \cite{Mar} proved that the whole mapping class group $\Mod(S_g)$ cannot be realized geometrically; i.e., $pr_g$ does not have a section. We have the following generalization to the infinite index subgroup $\mathcal{H}_g$.
\begin{cor}\label{nielsen}
The projection $pr_g$ does not have a section over the subgroup $\mathcal{H}_g$ for $g\ge 2$. In particular, $pr_g$ has no section for $g\ge 2$. 
\end{cor}
This extends the result of Markovic--Saric \cite{MS} that $\mathcal{H}_2$ cannot be realized geometrically and also gives a new proof of Markovic \cite{Mar} that mapping class group cannot be realized.

\para{Historic remark}
The Nielsen realization problem for $S_{g;m_1,...,m_r}^b$ asks if there exists a section of $pr_{g;m_1,...,m_r}^b$ over a subgroup of $\Mod(S_{g;m_1,...,m_r}^b)$. Nielsen (1943) posed this question for finite subgroups first and Kerckhoff \cite{Kerk} showed that a lift always exists for finite subgroups of $\Mod(S_g)$. The first result on the Nielsen realization problem for the whole mapping class group is a theorem of Morita \cite{Mor} that there is no section for the projection $\Diff^2_+(S_g)\to \Mod(S_g)$ when $g\ge 18$. Then Markovic \cite{Mar} (further extended by Markovic--Saric \cite{MS} on the genus bound)  showed that $pr_g$ does not have a section for $g\ge 2$. Franks--Handel \cite{FH}, Bestvina--Church--Suoto \cite{BCS} and Salter--Tshishiku \cite{ST} also obtained the non-realization theorems for $C^1$ diffeomorphisms. We refer the readers to the survey paper of Mann--Tshishiku \cite{MT} for more history and previous ideas.

\para{Idea of the proof} Our proof essentially uses \emph{torsion elements} (i.e., finite-order elements) of the corresponding mapping class group. The main observation is that the torsion elements in mapping class groups are not compatible with each other. By the Ahlfors' trick, which states that a torsion element in a mapping class group has a unique realization up to conjugation, we reach a contradiction by finding a global fixed point. To make use of our argument on a torsion-free group like the braid group $\Mod(S_{0;n}^1)$, we use the minimal decomposition theory of Markovic \cite{Mar} to modify the realization and apply the same strategy.

\para{Connection with Markovic's work \cite{Mar}}
To prove that $pr_{0;6}$ and $pr_{0;6,1}$ have no sections, we only use the the group structure and the Ahlfors' trick. The difficulty in other cases like $pr_{0;n}^1$ is the lack of torsion elements. For example, the braid group $\Mod(S_{0;n}^1)$ is torsion-free. Markovic's minimal decomposition theory gives us a tool to modify the action to obtain finite action. This is one of the novelty of this paper. 

The difference between our work and \cite{Mar}, \cite{MS} lies in the final contradiction. They used many relations like braid relation, chain relation and directly use the Ahlfors' trick on torsion elements. We only make use of two special torsion elements. However instead of directly having torsion elements, we have to make torsion elements appear by applying the minimal decomposition theory. The proof in this paper is conceivably much simpler.

\para{Structure of the paper} In Section 2, we give a local argument showing that the projection $pr_{0;1,6}$ does not have a section using torsion elements. In Section 3, we define minimal decomposition and prove Theorem \ref{main}, Corollary \ref{nielsen} by using a technical theorem which is a consequence of the minimal decomposition theory. We then prove the technical theorem in Section 4.

\para{Acknowledgements} This project obtained ideas from a previous paper with Nick Salter \cite{CS} about torsion elements of spherical braid group. She thanks Benson Farb, Nick Salter and Bena Tshishiku for asking the question about lifting braid group in Oberwalfach 2016 conference on surface bundles; she thanks Benson Farb, Dan Margalit and Nick Salter for discussions and comments on the paper. She would also like to thank Vlad Markovic for very useful discussions and the anonymous referee for suggestions on the paper. 
 
\section{a local argument}
In this section, we give a local argument showing that the projection $pr_{0;1,6}$ and $pr_{0;6}$ do not have sections. The following is an old theorem of Ahlfors on the uniqueness of Nielsen realization for finite subgroups; see, e.g., Markovic \cite[Proposition 1.1]{Mar}. Let $f\in \Mod(S_{g;m_1,...,m_r}^b)$ be a finite order mapping class. A homeomorphism representative of $f$ is a finite order element $h\in \Homeo_+(S_{g;m_1,...,m_r}^b)$ such that $h$ is homotopic to $f$ and has the same order as $f$.
\begin{prop}[Ahlfors' trick]
Let $f\in \Mod(S_{g;m_1,...,m_r}^b)$ be a finite order mapping class, then $f$ has a unique homeomorphism representative up to conjugation in $\Homeo_+(S_{g;m_1,...,m_r}^b)$.
\end{prop}
In the following, we only need the genus $0$ case of the Ahlfors' trick, which goes back to Brouwer \cite{Br}, Eilenberg \cite{Ei} and Ker{\'e}kj{\'a}rt{\'o} \cite{Ker2}; see also Constantin--Kolev \cite{CK}. For $pr_{0;6}$ and $pr_{0;6,1}$, we have the following argument.

\begin{thm}\label{local}
The projections 
\[
pr_{0;6,1}: \Homeo_+(S_{0;6,1})\to \Mod(S_{0;6,1})\text{ and }pr_{0;6}: \Homeo_+(S_{0;6})\to \Mod(S_{0;6})\] do not have sections.
\end{thm}
\begin{proof}
The above non-existence follows from the incompatibility of finite order elements in $\Mod(S_{0;6,1})$ and $\Mod(S_{0;6}^1)$. We prove the $pr_{0;6,1}$ case first. We name the marked points $p_0,p_1,...,p_6$ for both $\Homeo_+(S_{0;6,1})$ and $\Mod(S_{0;6,1})$ where $p_0$ is the point that is fixed globally. We consider the following two torsion elements in $\Mod(S_{0;6,1})$: 
\begin{itemize}
\item $\alpha_1$: the rotation of order $6$ fixing $p_0$ and no other marked points, 
\item $\alpha_2$: the rotation of order $5$ fixing $p_0$ and $p_6$. 
\end{itemize}
Now we assume that there exists a section $s$ of $pr_{0;6,1}: \Homeo_+(S_{0;1,6})\to \Mod(S_{0;6,1})$. 

By the Ahlfors' trick, finite order element of $\Homeo_+(S_0)$ is conjugate to an actual rotation. Then $s(\alpha_1)$ has another fixed point other than $p_0$, we call this point $A$. We know that $A$ is not a marked point because $\alpha_1$ fixes no other marked points. The goal of the proof is to show that $A$ is a global fixed point for $\Mod(S_{0;6,1})$, which contradicts the fact that $s(\alpha_2)$ only fixes $p_0,p_6$ but not $A$. This follows from the Ahlfors' trick on $s(\alpha_2)$.

For $0<k<6$, since $s(\alpha_1^k)$ is a nontrivial rotation, we know that $\fix(s(\alpha_1^k))=\{p_0,A\}$. If $g\in \Mod(S_{0;6,1})$ commutes with $\alpha_1^k$, then 
\[
s(g)(\{p_0,A\})=s(g)(\fix(s(\alpha_1^k)))=\fix(s(g\alpha_1^kg^{-1}))=\fix(s(\alpha_1^k))=\{p_0,A\}.
\]  
Since we also know that $s(g)$ fixes $p_0$, we obtain that $s(g)$ fixes $A$.  Denote by $C(k)$ the centralizer of $\alpha_1^k$ in $\Mod(S_{0;6,1})$. The above discussion establishes the fact that $s(C(k))$ fixes $A$. Denote by $G<\Mod(S_{0;6,1})$ the subgroup generated by $C(2)$ and $C(3)$. To finish our proof, all we need now is to show that $G=\Mod(S_{0;6,1})$. 

Let $\sigma_1$ be the half twist in $\Mod(S_{0;6,1})$ and $\alpha_1$ be the rotation as in the following figure. 

\begin{figure}[H]
\centering
\begin{tikzpicture}
\draw[\discColor, thick] (0,0) circle [radius = 1*\discSizeX];

\node at (0:0.9*\discSizeX) {$p_1$};
\node (1) at (0:0.67*\discSizeX) {};
\fill[black, thick] (0:0.67*\discSizeX) circle [radius = 0.05*\discSizeX];
\node at (60:0.9*\discSizeX) {$p_2$};
\node (2) at (60:0.67*\discSizeX){};
\fill[black, thick] (60:0.67*\discSizeX) circle [radius = 0.05*\discSizeX];
\draw [->, thick] (1) to [out=30,in=0] (2);
\draw [->, thick] (2) to [out=225,in=145] (1);
\node at (30:0.65*\discSizeX) {$\sigma_1$};

\draw [->, thick] (0:1.1*\discSizeX) to [out=90,in=330] (30:1.1*\discSizeX);
\node at (25:1.3*\discSizeX) {$\alpha_1$};
\node at (10:1.25*\discSizeX) {$\frac{2\pi}{6}$};

\node at (120:0.9*\discSizeX) {$p_3$};
\fill[black, thick] (120:0.67*\discSizeX) circle [radius = 0.05*\discSizeX];
\node at (180:0.9*\discSizeX) {$p_4$};
\fill[black, thick] (180:0.67*\discSizeX) circle [radius = 0.05*\discSizeX];
\node at (240:0.9*\discSizeX) {$p_5$};
\fill[black, thick] (240:0.67*\discSizeX) circle [radius = 0.05*\discSizeX];
\node at (300:0.9*\discSizeX) {$p_6$};
\fill[black, thick] (300:0.67*\discSizeX) circle [radius = 0.05*\discSizeX];
\end{tikzpicture}
\caption{The mapping class $\sigma_1$ and $\alpha_1$ in $\Mod(S_{0;6,1})$}
\end{figure}
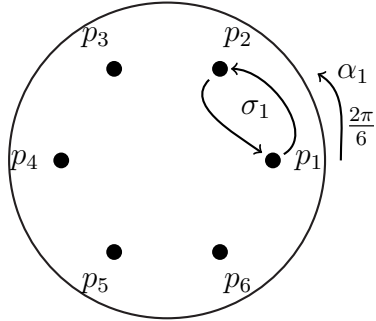
Define $\sigma_i:=\alpha_1^{-i}\sigma_1\alpha_1^i$, which is also a half twist. First of all, $\Mod(S_{0;6,1})$ is generated by $\sigma_1,...,\sigma_5$. This can be seen from the fact that the braid group $\Mod(S_{0;6}^1)$ is already generated by $\sigma_1,...,\sigma_5$ (see, e.g., \cite[Page 246]{FM}) and that $\Mod(S_{0;6,1})$ is the quotient of $\Mod(S_{0;6}^1)$ by the Dehn twist about the boundary component. Therefore, we know that $\sigma_i$ and $\alpha_1$ generate $\Mod(S_{0;6,1})$. 

Since $\alpha_1\in C(2)$, all we need to prove is that $\sigma_3\in G$. We prove this by explicitly writing $\sigma_3$ as a product of elements in $C(2)$ and $C(3)$. By observation, $\sigma_1\sigma_4,\sigma_2\sigma_5, \sigma_3\sigma_6 \in C(3)$ and $\sigma_1\sigma_3\sigma_5,\sigma_2\sigma_4\sigma_6\in C(2)$. We now start with 
\[\alpha_1=\sigma_1\sigma_2\sigma_3\sigma_4\sigma_5 \in G.\]
Since $\sigma_5\sigma_2\in G$, we have that 
\[
\sigma_1\sigma_2\sigma_3\sigma_4\sigma_5(\sigma_5\sigma_2)^{-1} \in G.\]
By commutativity of $\sigma_2$ and $\sigma_4$, we obtain
\[
\sigma_1\sigma_2\sigma_3\sigma_2^{-1}\sigma_4 \in G.\]
Applying the same calculation for $\sigma_1\sigma_4\in G$, we obtain
\[
\sigma_1\sigma_2\sigma_3\sigma_2^{-1}\sigma_1^{-1}\in G.
\]
Since $\sigma_1\sigma_3\sigma_5\in G$, we obtain
\[
(\sigma_1\sigma_3\sigma_5)^{-1}\sigma_1\sigma_2\sigma_3\sigma_2^{-1}\sigma_1^{-1}(\sigma_1\sigma_3\sigma_5)\in G.
\]
But we know that $\sigma_5$ commutes with every other element in the above equation, so we obtain
\[
\sigma_3^{-1}\sigma_2\sigma_3\sigma_2^{-1}\sigma_3\in G.\]
Since $\sigma_3\sigma_6\in G$, we obtain
\[
(\sigma_3\sigma_6)\sigma_3^{-1}\sigma_2\sigma_3\sigma_2^{-1}\sigma_3(\sigma_3\sigma_6)^{-1}\in G.\]
But we know that $\sigma_6$ commutes with every other element in the above equation, so we obtain
\[\sigma_2\sigma_3\sigma_2^{-1}\in G.\]
Since $\sigma_2\sigma_5\in G$, we obtain
\[(\sigma_2\sigma_5)^{-1}\sigma_2\sigma_3\sigma_2^{-1}(\sigma_2\sigma_5)\in G.\]
But we know that $\sigma_5$ commutes with every other element in the above equation, so we obtain
\[\sigma_3\in G.\] 
This concludes the proof for case $pr_{0;6}^1$

For case $pr_{0;6}$, we assume that $pr_{0;6}$ has a section $s$.
We name the marked points $p_1,...,p_6$ for both $\Homeo_+(S_{0;6})$ and $\Mod(S_{0;6})$. We consider the following two torsion elements in $\Mod(S_{0;6})$: 
\begin{itemize}
\item $\alpha_1$: the rotation of order $6$ fixing no marked points, 
\item $\alpha_2$: the rotation of order $5$ fixing $p_6$. 
\end{itemize}
By the Ahlfors' trick, finite order element of $\Homeo_+(S_0)$ is conjugate to an actual rotation. Then $s(\alpha_1)$ has two fixed points $A,B$. The goal of the proof is to show that the set $\{A,B\}$ is globally preserved by $s(\Mod(S_{0;6}))$, which contradicts the fact that $s(\alpha_2)$ cannot fix the set $\{A,B\}$. If $s(\alpha_2)$ fixes the set $\{A,B\}$ then since the order of $s(\alpha_2)$ is odd, $s(\alpha_2)$ fixes $A,B$ point-wise. Therefore $s(\alpha_2)$ fixes $p_6,A,B$ which contradicts to the Ahlfors' trick that $s(\alpha_2)$ is an actual rotation.

For $0<k<6$, since $s(\alpha_1^k)$ is a nontrivial rotation, we know that $\Fix(s(\alpha_1^k))=\{A,B\}$. If $g\in \Mod(S_{0;6})$ commutes with $\alpha_1^k$, then 
\[
s(g)(\{A,B\})=s(g)(\fix(s(\alpha_1^k)))=\fix(s(g\alpha_1^kg^{-1}))=\fix(s(\alpha_1^k))=\{A,B\}.
\]  
We denote by $C(k)$ the centralizer of $\alpha_1^k$ in $\Mod(S_{0;6})$. The above discussion establishes the fact that $s(C(k))$ preserves the set $\{A,B\}$. We denote by $G<\Mod(S_{0;6})$, the subgroup generated by $C(2)$ and $C(3)$. To finish our proof, all we need now is to show that $G=\Mod(S_{0;6})$ which follows the same computation as in case $pr_{0;6,1}$.
\end{proof}
\begin{rem}
Notice that the above argument does not give any information for the case of 
\[pr_{0;n,1}:\Homeo_+(S_{0;n,1})\to \Mod(S_{0;n,1})\] when $n$ is a prime number. We need a stronger tool to deal with the general case.
\end{rem}

\section{general case}
In this section, we prove Theorem \ref{main}. The local argument shows that the section of $pr_{0;6}$ and $pr_{0;6,1}$ do not exist. For $n\ge 6$, assume that $pr_{0:n}^1$ has a section $\mathcal{E}$. Let $c$ be a simple closed curve in $S_{0;n}^1$ that surrounds $6$ points. Let $em: \Mod(S_{0;6}^1)\to \Mod(S_{0;6,n-6}^1)$ be the embedding of the subgroup that consists of mapping classes that are the identity map outside of $c$. Then we have the following compositions of maps
\[
\rho: \Mod(S_{0;6}^1)\xrightarrow{\text{em}} \Mod(S_{0;6,n-6}^1)\xrightarrow{\mathcal{E}} \Homeo_+(S_{0;6,n-6}^1)\xrightarrow{\text{forget}}\Homeo_+(S_{0;6}^1)\xrightarrow{\text{pinch}} \Homeo_+(S_{0;6,1}),
\]
where ``forget" denotes the forgetful map forgetting the extra $n-6$ marked points and ``pinch" denotes the action on the the quotient space $S_{0;6}^1/\sim$ that identifies the boundary component. 

By definition, the homomorphism $\rho$ is almost a realization of $pr_{0:6,1}$ except that the center element of $\Mod(S_{0;6}^1)$  (the Dehn twist $T_b$ about the boundary component $b$) does not map to the identity homeomorphism. We solve this problem by the minimal decomposition theory established by Markovic \cite{Mar}. The key idea is that the center element is canonically semi-conjugate to identity. 

\subsection{Minimal decomposition}
In this section, we recall a theory called minimal decomposition of surface homeomorphisms. This is established in the celebrated paper of Markovic \cite{Mar} giving the first proof that the mapping class group cannot be geometrically realized as homeomorphisms. We apply Markovic's theory to modify the homomorphism $\rho$ to an actual section of $pr_{0;6,1}$.

We recall the definition of upper semi-continuous decomposition of a surface; see also Markovic \cite[Definition 2.1]{Mar}. Let $M$ be a surface. 
\begin{defn}[Upper semi-continuous decomposition]
Let $\mathbf{S}$ be a collection of closed, connected subsets of $M$. We say that $\mathbf{S}$ is an upper semi-continuous decomposition of $M$ if the following holds:
\begin{itemize}
\item If $S_1,S_2\in \mathbf{S}$, then $S_1\cap S_2=\emptyset$.
\item If $S\in \mathbf{S}$, then $S$ does not separate $M$; i.e., $M-S$ is connected.
\item We have $M=\cup_{S\in \mathbf{S}} S$.
\item If $S_n\in \mathbf{S}, n\in \mathbb{N}$ is a sequence that has the Hausdorff limit $S_0$ then there exists $S\in\mathbf{S}$ such that $S_0\subset S$.
\end{itemize}
\end{defn}
Now we define acyclic sets on a surface.
\begin{defn}[Acyclic sets]
Let $S\subset M$ be a closed, connected subset of $M$ which does not separate $M$. We say that $S$ is acyclic if there is a simply connected open set $U\subset M$ such that $S \subset U$ and $U-S$ is homeomorphic to an annulus.
\end{defn}
The easiest examples of an acyclic set are a point, an embedded closed arc or an embedded closed disk in $M$. Let $S\subset M$ be a closed, connected set that does not separate M. Then $S$ is acyclic if and only if there is a lift of $S$ to the universal cover $\widetilde{M}$ of $M$ which is a compact subset of $\widetilde{M}$. The following theorem is a classical result called the Moore's theorem; see, e.g., \cite[Theorem 2.1]{Mar}. The Moore's theorem is used to modify $\rho$.
\begin{thm}[Moore's theorem]\label{moore}
Let $M$ be a surface and $\mathbf{S}$ be an upper semi-continuous decomposition of $M$ so that every element of $\mathbf{S}$ is acyclic. Then there is a continuous map $\phi:M\to M$ that is homotopic to the identity map on $M$ and such that for every $p\in M$, we have that $\phi^{-1}(p)\in \mathbf{S}$. Moreover we have that $\mathbf{S}=\{\phi^{-1}(p)|p\in M\}$.
\end{thm}
We now recall the minimal decomposition theory. The following definition is \cite[Definition 3.1]{Mar}
\begin{defn}[Admissible decomposition]
Let $\mathbf{S}$ be a upper semi-continuous decomposition of $M$. Let $G$ be a subgroup of $\Homeo(M)$. We say that $\mathbf{S}$ is admissible for the group $G$ if the following holds:
\begin{itemize}
\item Each $f\in G$ preserves set-wise every element of $\mathbf{S}$.
\item Let $S\in \mathbf{S}$. Then every point, in every frontier component of the surface $M-S$ is a limit of points from $M-S$ that belongs to acyclic elements of $\mathbf{S}$.
\end{itemize}
If $G$ is a cyclic group generated by a homeomorphism $f:M\to M$ we say that $S$ is an admissible decomposition of $f$.
\end{defn}
An admissible decomposition for $G<\Homeo(M)$ is called minimal if it is contained in every admissible decomposition for $G$. We have the following theorem from Markovic \cite[Theorem 3.1]{Mar}.
\begin{thm}[Existence of minimal decomposition]
Every group $G<\Homeo(M)$ has a unique minimal decomposition.
\end{thm}
Let $b$ be the boundary component of $S_{0;6}^1$ and $T_b$ be the Dehn twist about $b$. The following theorem is a modified version of Markovic \cite[Lemma 5.1]{Mar} for our case.
\begin{thm}\label{key}
Every element of the minimal decomposition $\mathbf{S}$ of $\rho(T_b)$ is acyclic and marked points belong to different elements of $\mathbf{S}$.
\end{thm}
To make the whole proof easier to follow, we postpone the proof of Theorem \ref{key} to the next section.
\subsection{The proof of Theorem \ref{main}}
Now we use Theorem \ref{key} to prove Theorem \ref{main}.
\begin{proof}
Let $\mathbf{S}$ be the minimal decomposition of $\rho(T_b)$. By Theorem \ref{key} and Theorem \ref{moore} (the Moore's theorem), the space $S_{0;6,1}/{\sim}$ is homeomorphic to $S_{0;6,1}$ where $x\sim y$ if and only if $x,y$ belong to the same element of $\mathbf{S}$. Since the minimal decomposition is canonical, if $f\in \Homeo_+(S_{0;6,1})$ commutes with $\rho(T_b)$, then $f$ permutes elements of $\mathbf{S}$. Therefore $f$ induces a homeomorphism of $S_{0,6,1}/{\sim}$. Since $T_b$ is the center of $\Mod(S_{0;6}^1)$, we obtain a new homomorphism $\rho({\sim}):  \Mod(S_{0;6}^1)\to \Homeo_+(S_{0,6,1}/{\sim})$ where $\rho({\sim})(T_c)=id$ by the definition of admissible decomposition. This is a section of $pr_{0:6,1}$, which contradicts the fact that  $pr_{0:6,1}$ has no section.

We now prove the cases $pr_{0;n}$ and $pr_{0;n,1}$, which is similar to the proof of case $pr_{0;n}^1$. For $n\ge 6$, assume that $pr_{0:n,1}$ or $pr_{0;n}$ has a section $\mathcal{E}$. Similarly we have the following compositions of maps
\[
\rho: \Mod(S_{0;6}^1)\xrightarrow{\text{em}} \Mod(S_{0;6,n-6})\xrightarrow{\mathcal{E}} \Homeo_+(S_{0;6,n-6})\xrightarrow{\text{forget}}\Homeo_+(S_{0;6}).\]
By the same argument as before, we obtain a homomorphism 
\[\rho(\sim): \Mod(S_{0;6,1})\to \Homeo_+(S_{0;6}).\] Even though $\rho(\sim)$ is not a realization that we have discussed, we still use the fixed point argument as the case $pr_{0;6}$ in the proof of Theorem \ref{local} to show that such $\rho(\sim)$ does not exist. We sketch the proof in the following.

Notice that the $6$ marked points in the domain of $\rho(\sim)$ correspond to the marked points in $\Homeo_+(S_{0;6})$. Therefore $\rho(\sim)(\alpha_1)$ fixes no marked points but two other points $A,B\in S_{0;6}$. By the same computation as in the proof of Theorem \ref{local}, we show that the whole group $\rho(\sim)(B_{0;6}^1)$ fixes $\{A,B\}$. However $\rho(\sim)(\alpha_2)$ has order $5$ and fixes $\{A,B\}$, which implies that $\rho(\sim)(\alpha_2)$ fixes $A,B$ point-wise. However $\rho(\sim)(\alpha_2)$ also fixes one marked point. This is a contradiction.
\end{proof}

\subsection{Application to Nielsen realization problem for closed mapping class group}
Now we proceed to apply Theorem \ref{main} to deal with Nielsen realization problem for $\mathcal{H}_g$. The same strategy has also been used in \cite{ST}.
\begin{proof}[Proof of Corollary \ref{nielsen}]
The subgroup $\mathcal{H}_g$ satisfies the following exact sequence
\[
1\to \mathbb{Z}/2\to \mathcal{H}_g\to \Mod(S_{0;2g+2})\to 1.
\]
Assume that $\mathcal{H}_g$ has a realization and that $\widetilde{\tau}\in \Homeo_+(S_g)$ is the realization of $\tau$. By the Ahlfors' trick, $\widetilde{\tau}$ is conjugate to the standard hyper-elliptic involution which means that $\widetilde{\tau}$ has $2g+2$ fixed point. Denote by $\Homeo_+(S_g)(\widetilde{\tau})$ the centralizer of $\widetilde{\tau}$. Thus we have the following exact sequence
\[
1\to \mathbb{Z}/2\to \Homeo_+(S_g)(\widetilde{\tau})\to \Homeo_+(S_{0;2g+2})\to 1.
\]
By Birman--Hilden theory \cite{BH} (see, e.g., \cite[Chapter 9.4]{FM}), we know that $\pi_0(\Homeo_+(S_g)(\widetilde{\tau}))=\mathcal{H}_g$. We have the following pullback diagram
\[
\xymatrix{
\Homeo_+(S_g)(\widetilde{\tau})\ar[r]\ar[d]^{pr(\mathcal{H}_g)} &\Homeo_+(S_{0;2g+2})\ar[d]^{pr_{0;2g+2}}\\
\mathcal{H}_g\ar[r]& \Mod(S_{0;2g+2}).}
\]
However, a section of $pr(\mathcal{H}_g)$ gives a section of $pr_{0;2g+2}$, which contradicts Theorem \ref{main}.
\end{proof}

\section{the proof of theorem \ref{key}}
To make the analysis easier, we take the following hyper-elliptic $\mathbb{Z}/2$ branched covers $p: S_{2;2}\to S_{0;6,1}$ and $p': S_2^2\to S_{0;6}^1$ so that we are working with a surface of genus $2$ with marked points and boundary components. 
\begin{figure}[h]
   \labellist 
  \small\hair 2pt
     \pinlabel $\tau$ at -10 67
 \pinlabel $P_1$ at 44 70
   \pinlabel $P_2$ at 60 50
      \pinlabel $P_3$ at 100 50
        \pinlabel $P_4$ at 120 50
      \pinlabel $P_5$ at 160 50
       \pinlabel $P_6$ at 170 70
        \pinlabel $p$ at 240 80
       \pinlabel $P_0^1$ at 225 130
          \pinlabel $P_0^2$ at 230 10
           \pinlabel $p_0$ at 350 140
             \pinlabel $p_1$ at 270 70
              \pinlabel $p_4$ at 430 70
               \pinlabel $p_2$ at 320 35
                  \pinlabel $p_3$ at 370 35
                   \pinlabel $p_5$ at 320 100
                     \pinlabel $p_6$ at 370 100
   \endlabellist
     \centerline{ \mbox{
 \includegraphics[width = 3.5in]{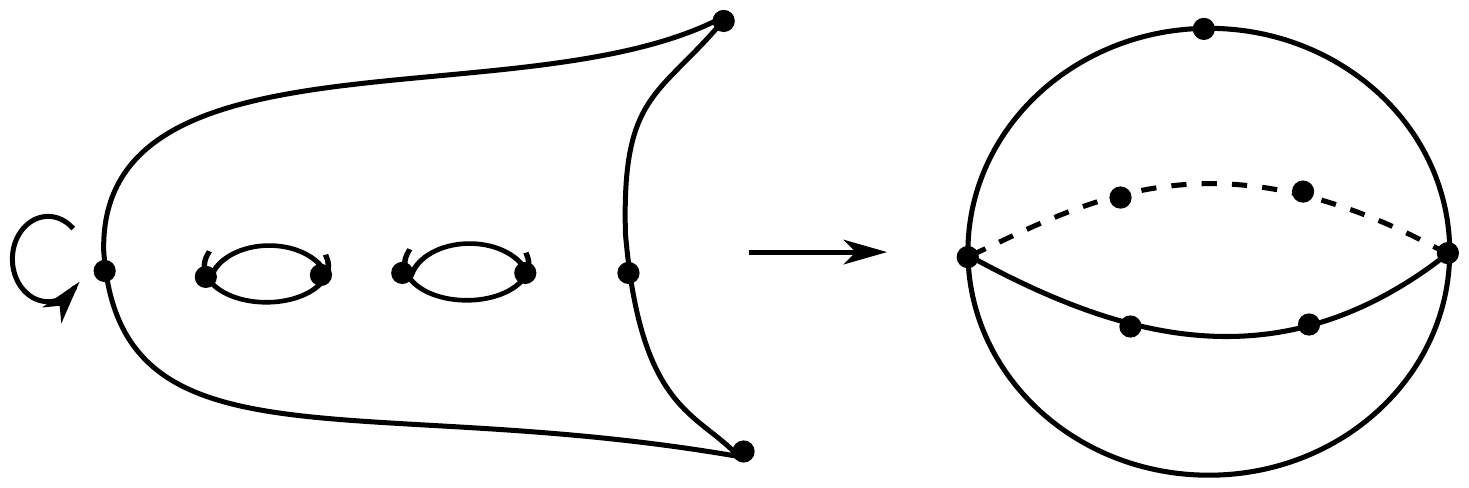}}}
 \caption{The projection $p: S_{2;2}\to S_{0;6,1}$}
 \label{p1}
  \end{figure}
  \begin{figure}[h]
   \labellist 
  \small\hair 2pt
     \pinlabel $\tau'$ at -10 67
 \pinlabel $P_1$ at 44 70
   \pinlabel $P_2$ at 60 50
      \pinlabel $P_3$ at 100 50
        \pinlabel $P_4$ at 120 50
      \pinlabel $P_5$ at 160 50
       \pinlabel $P_6$ at 170 70
        \pinlabel $p$ at 240 80
       \pinlabel $b^1$ at 225 126
          \pinlabel $b^2$ at 225 10
           \pinlabel $b$ at 350 140
             \pinlabel $p_1$ at 295 70
              \pinlabel $p_4$ at 390 70
               \pinlabel $p_2$ at 320 35
                  \pinlabel $p_3$ at 365 35
                   \pinlabel $p_5$ at 315 100
                     \pinlabel $p_6$ at 360 100
   \endlabellist
     \centerline{ \mbox{
 \includegraphics[width = 3.5in]{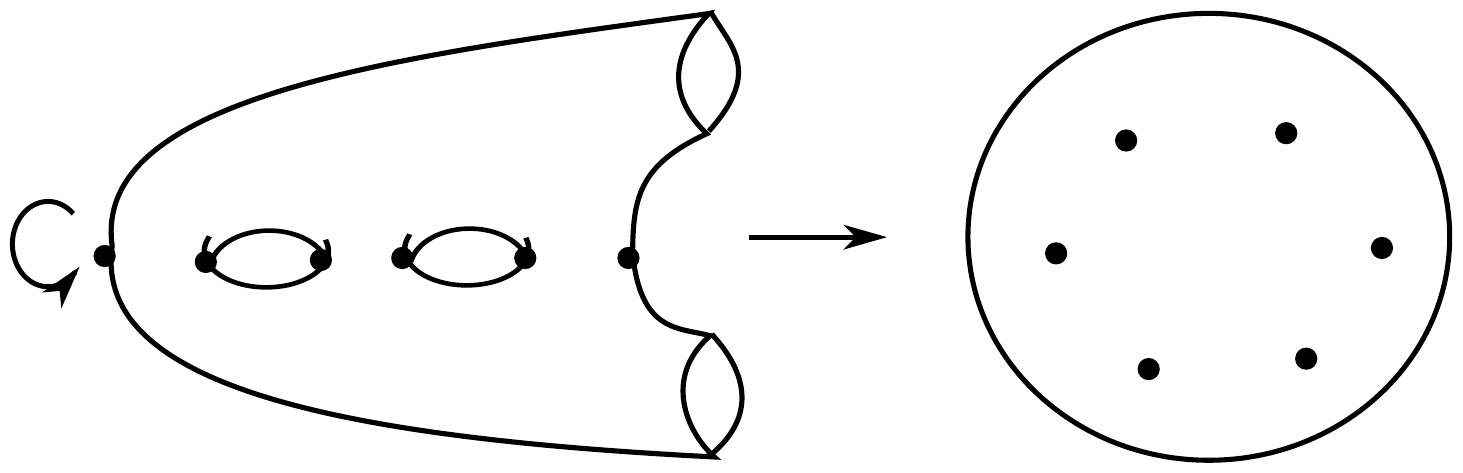}}}
 \caption{The projection $p': S_{2}^2\to S_{0;6}^1$}
 \label{p2}
  \end{figure}
  
Let $\tau$ and $\tau'$ be the corresponding hyper-elliptic involution of $S_{2;2}$ and $S_2^2$. We use the same letter to represent both a homeomorphism and its mapping class. We also use the same letter to represent marked points in $S_2^2$ and $S_{2;2}$ and marked points in $S_{0;6}^1$ and $S_{0;6,1}$ as in Figure \ref{p1} and \ref{p2}. Let $\Homeo(S_{2,2})(\tau)$ and $\Mod(S_{2}^2)(\tau')$ be the centralizer of $\tau$ and $\tau'$. We have the following two short exact sequences
\[
1\to \mathbb{Z}/2\to \Mod(S_{2}^2)(\tau')\to \Mod(S_{0;6}^1)\to 1
\]
and 
\[
1\to \mathbb{Z}/2\to \Homeo(S_{2,2})(\tau)\to \Homeo(S_{0;6,1})\to 1,
\]
The homomorphism $\rho: \Mod(S_{0;6}^1)\to \Homeo(S_{0;6,1})$ induces a homomorphism $\rho' :\Mod(S_{2}^2)(\tau')\to \Homeo(S_{2,2})(\tau)$. Let $b^1,b^2$ be the two boundary components of $S_{2}^2$ and denote $F:=\rho'(T_{b^1}T_{b^2})$, which is a lift of $\rho(T_b)\in \Mod(S_{0;6}^1)$. Let $\mathbf{S}'$ be the minimal decomposition of $F$. Since $F$ commutes with $\tau$, we know that $\mathbf{S}'$ is $\tau$ invariant. Since $F$ is a lift of $\rho(T_b)$, we know that $p(\mathbf{S}')$ is an admissible decomposition of $\rho(T_b)$. To prove that the admissible decomposition of $\rho(T_b)$ satisfies Theorem \ref{key}, we only need to show that $p(\mathbf{S}')$ satisfies Theorem \ref{key}. 

Let $e: S_{1}^1\subset S_{2}^2$ be the following embedding and $c$ be the boundary of the subsurface $e(S_1^1)$.

\begin{figure}[h]
   \labellist 
  \small\hair 2pt
     \pinlabel $c$ at 75 125
     \pinlabel $e(S_1^1)$ at 25 78
      \pinlabel $S_2^2$ at 140 78
   \endlabellist
     \centerline{ \mbox{
 \includegraphics[width = 2in]{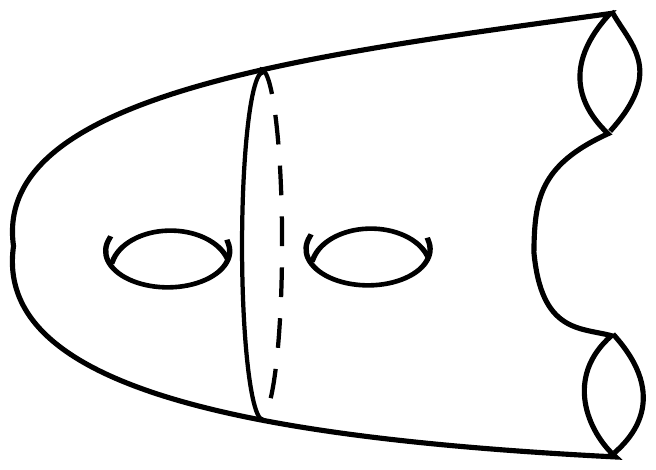}}}
 \caption{The embedding $e:S_1^1\to S_2^2$}
  \end{figure}
  
\begin{lem}
The induced map of $e$ on mapping class groups $E: \Mod(S_{1}^1)\to \Mod(S_{2}^2)$ has image in $\Mod(S_{2}^2)(\tau')$.
\end{lem}
\begin{proof}
It is classical that $\Mod(S_{1}^1)$ commutes with the elliptic involution. Therefore the embedding image $E(\Mod(S_{1}^1))$ commutes with the hyper-elliptic involution $\tau$. See \cite[Page 75-77]{FM} about centers of mapping class groups.
\end{proof}
Therefore, we obtain the following theorem which is the same as \cite[Theorem 4.1]{Mar}.
\begin{thm}\label{acyclic}
There exists an admissible decomposition of $S_{2;2}$ for $F$ with the following property: there exists a simple closed curve $\alpha$ homotopic to $c$ such that if $p\in S_{2;2}$ belongs to the torus minus a disc (which is one of the two components obtained after removing $\alpha$ from $S_{2,2}$), then the element of the decomposition that contains $p$ is acyclic. 
\end{thm}
\begin{proof}[Sketch proof]
We use the same Anosov map $A'$ on $2$-torus as in \cite[Theorem 4.1]{Mar} and blow it up at the fixed point and extend to identity outside of $e(S_1^1)$ to obtain $A\in \Homeo(S_{2}^2)$. Let $[A]$ be the corresponding mapping class. We know that $[A]$ commutes with $T_{b^1}T_{b^2}$ and $T_{b^1}T_{b^2}$ is identity on the subspace $S_{1}^1$. By the global shadowing property of Anosov flow and \cite[Lemma 4.14]{Mar}, the homeomorphism $F$ setwise preserves each element of the corresponding decomposition of $\rho'([A])$ which has the property stated in the theorem. 
\end{proof}
We now prove the following lemma which is similar to \cite[Lemma 5.1]{Mar}.
\begin{lem}
The minimal decomposition of $F$ consists of acyclic elements.
\end{lem}
\begin{proof}
The set of all points $p\in S_2$ such that the corresponding element $S_p\in \mathbf{S}'$ is acyclic is denoted by $M_F$. By the definition of minimal decomposition, $x\in M_F$ if and only if there exists an admissible decomposition such that $x$ belong to an acyclic element. Therefore $M_F$ contains the torus minus a disc in Theorem \ref{acyclic}. Let $M_F'$ be the connected component of $M_F$ that contains this torus minus a disc. By \cite[Proposition 2.1]{Mar}, the subset $M_F'$ is an open subsurface with finitely many ends. 

If $M_F'\neq S_2$, then let $\beta_n$ be a nested sequence that determines one end $K$ of $M_F'$. By Theorem \ref{acyclic}, there exists a simple closed curve $\gamma\subset M_F'$ such that $\gamma$ is homotopic to $c$. Since the center of $\Mod(S_2)$ is generated by hyper-elliptic involution $\tau$, every curve in $S_2$ has a $\tau$-invariant representative. 

Let $\delta'$ be a simple closed curve in $S_{2}$ such that $i(\delta',\gamma)\neq 0$ and $i(\delta', \beta_n)\neq 0$ where $i(\_,\_)$ denotes the geometric intersection number. Find a $\tau$-invariant representative $\delta$ of $\delta'$ that avoids $b^1,b^2$. Then the mapping class $T_\delta\in \Mod(S_2^2)$ satisfies that $i(T_\delta(\gamma),\beta_n)\neq 0$ and  $i(T_\delta(\gamma),\gamma)\neq 0$ on $S_2$.

Since $\rho'(\Mod(S_2^2))$ commutes with $F$, we know that $\rho'(\Mod(S_2^2))$ permutes connected components of $M_F$. Therefore $\rho'(T_\delta)(\gamma)$ is either contained in $M_F'$ or is disjoint from $M_F'$. However $i(T_\delta(\gamma),\gamma)\neq 0$ rules out the possibility that $\rho'(T_\delta)(\gamma)$ is disjoint from $M_F'$. Therefore $\rho'(T_\delta)(\gamma)\subset M_F'$. This contradicts to the fact that $\rho'(T_\delta)$ intersect each curve $\beta_n$ in the nested sequence converging to one end $K$. 
\end{proof}
\begin{lem}
Marked points do not belong to the same element in $p(\mathbf{S}')$.
\end{lem}
\begin{proof}
Points are named in Figure \ref{p1} and \ref{p2}. If two marked points belong to one element in $p(\mathbf{S}')$, we claim that there exists $S\in \mathbf{S}'$ such that $P_i,P_j\in S$ for $1\le i\neq j\le 6$. If there exists $S\in \mathbf{S}'$ such that $P_0^1,P_i\in S$ for $1\le i\le 6$ and $S\in \mathbf{S}'$, then since $\tau$ permutes elements of $\mathbf{S}'$, we know that $P_0^2,P_i\in S$. Let $f\in \Mod(S_2^2)(\tau)$ be a mapping class that permutes $P_i,P_j$. Since $\rho'(f)$ preserves the set $\{P_0^1,P_0^2\}$, we know that $\rho'(f)$ preserves $S$ as well. Therefore $P_i,P_j\in S$ for $1\le i\neq j\le 6$. Therefore without loss of generality, we assume that there exists $S\in \mathbf{S}'$ such that $P_1,P_2\in S$.

Since $\tau$ preserves the minimal decomposition $\mathbf{S}'$ and $S$ is acyclic, there exists an open neighborhood $U$ of $S$ consisting of elements of $\mathbf{S}'$ such that $U$ is simply-connected and non-separating. Denote by $U(S)$ the connected component of $U\cap \tau(U)$ that contains $S$. Since $\tau$ permutes elements in $\mathbf{S}'$ and $U$ consists of elements in $\mathbf{S}'$, we know that $\tau(U)\cap U$ consists of elements in  $\mathbf{S}'$. Since element in $\mathbf{S}'$ is connected, a connected component of $\tau(U)\cap U$ also consists of elements in $\mathbf{S}'$. Since each connected component of the intersection of two simply connected open subset on a surface is also simply connected, we know that $U(S)$ is simply connected, open, a union of elements in $\mathbf{S}'$ and satisfies that $S\subset U(S)$ and $\tau(U(S))=U(S)$. 

Since $S$ contains $P_1,P_2$, we know that $U(S)$ contains $P_1,P_2$ as well. Therefore the projection $p(U(S))$ contains a simple arc $\alpha$ connecting $p_1,p_2$. Since $U(S)$ is open, we can choose the arc $\alpha$ such that $\alpha$ does not pass other marked points. However, the pre-image of $\alpha$ is a nontrivial loop. This can be seen from the fact that first of all, the pre-image only depends on the isotopy type of the arc $\alpha$. At least one simple arc connecting two marked points has pre-image a nontrivial loop. Since $\Mod(S_{0;6}^1)$ acts transitively on simple arcs connecting two marked points by change of coordinate principle \cite[Chapter 1.3]{FM}, we know that the pre-image of $\alpha$ is nontrivial as well. 
\end{proof}
Now we have all we need to prove Theorem \ref{key}
\begin{proof}[Proof of Theorem \ref{key}]
Since the image of each element of $\mathbf{S}'$ under $p$ is also connected and closed, we know that the minimal decomposition of $\rho(T_b)$ satisfies the property as stated in Theorem \ref{key} because $p(\mathbf{S}')$ does.
\end{proof}

    	\bibliography{braidlifting}{}
	\bibliographystyle{alpha}
	
California Institute of Technology 

Department of Mathematics 

Pasadena, CA 91125, USA

E-mail: chenlei1991919@gmail.com

\end{document}